\documentclass[11pt]{article}
\usepackage{amsmath,amsfonts,amsthm,amssymb,graphicx,bm,bbm}
\usepackage{calrsfs}

\usepackage{caption}
\usepackage{subcaption}
\usepackage{xcolor}
\usepackage[font=small,labelfont=bf]{caption}
\usepackage{algorithm}
\usepackage{algorithmic}

\theoremstyle{plain}
\newtheorem{theorem}{Theorem}

\newtheorem{lemma}{Lemma}

\begin{document}

\title{Minimax Efficient Finite-Difference Stochastic Gradient Estimators Using Black-Box Function Evaluations}
\author{Henry Lam\thanks{Department of Industrial Engineering and Operations Research, Columbia University, 500 W 120th St., New York, NY 10027, USA. Email: {\tt henry.lam@columbia.edu}}
\and
Haidong Li\thanks{Department of Industrial Engineering and Management, Peking University, No. 60 Yannan Yuan, Beijing, China. Email: {\tt haidong.li@pku.edu.cn}}
\and
Xuhui Zhang\thanks{Department of Management Science and Engineering, Stanford University, 475 Via Ortega, Stanford, CA 94305, USA. Email: {\tt xuhui.zhang@stanford.edu}}}
\date{}
\maketitle

\begin{abstract}

Standard approaches to stochastic gradient estimation, with only noisy black-box function evaluations, use the finite-difference method or its variants. While natural, it is open to our knowledge whether their statistical accuracy is the best possible. This paper argues so by showing that central finite-difference is a nearly minimax optimal zeroth-order gradient estimator for a suitable class of objective functions and mean squared risk, among both the class of linear estimators and the much larger class of all (nonlinear) estimators.

\end{abstract}



\section{Introduction}
Stochastic gradient estimation is of central importance in simulation analysis and optimization. It concerns the estimation of gradients under noisy environments driven by data or Monte Carlo simulation runs. This problem arises as a key ingredient in sensitivity analysis and uncertainty quantification for simulation models, descent-based algorithms in stochastic optimization and machine learning, and other applications such as financial portfolio hedging. For an overview of stochastic gradient estimation and its applications, see, e.g., \cite{l1991overview}, \cite{fu2006gradient}, \cite{glasserman2013monte} Chapter 7 and \cite{asmussen2007stochastic} Chapter 7.

In this paper, we consider stochastic gradient estimation when only a noisy simulation oracle to evaluate the function value or model output is available. This corresponds to black-box settings in which it is costly, or even impossible, to utilize the underlying dynamics of a simulation model, or to distort the data collection mechanism in an experiment given the input. In stochastic optimization, such an oracle is also known as the zeroth-order \cite{ghadimi2013stochastic,nesterov2017random}. These gradient estimators are in contrast to unbiased estimators obtained from methods such as the infinitesimal perturbation analysis \cite{heidelberger1988convergence,ho1983infinitesimal}, the likelihood ratio or score function method \cite{glynn1990likelihood,rubinstein1986score,reiman1989sensitivity}, measure-valued differentiation \cite{heidergott2010gradient} and other variants that require structural information on model dynamics.

Under the above setting, the most natural and common approach is to use the finite-difference (FD) method \cite{zazanis1993convergence,fox1989replication,frolov1963calculation}. This entails simulating the function values at two neighboring inputs and using the first principle of differentiation to approximate the derivative. The resulting estimator has a bias coming from this derivative approximation, on top of the variance coming from the function evaluation noise. This leads to a subcanonical overall mean squared error (MSE), i.e., with a decay rate slower than the order $O(n^{-1})$ exhibited by standard central limit theorems, where $n$ refers to the number of differencing pairs in the simulation. This also leads to a need to rightly tune the perturbation size between the two input values. It is widely known (e.g., \cite{l1991overview} Section 3.2) that, for twice continuously differentiable functions, the optimal attainable MSE for central finite-difference (CFD) schemes is of order $O(n^{-2/3})$. On the other hand, when one uses forward or backward finite-differences, the optimal MSE deteriorates to $O(n^{-1/2})$.

Although the optimal MSEs \emph{within} the classes of FD schemes are well-known, a question arises whether these classes are optimal or better compared to other, possibly much larger, classes of gradient estimators. Our goal in this paper is to give a first such study on the optimality on a class level.

Our main results show that, under a general setting, CFD is nearly optimal among any possible gradient estimation schemes. This optimality is in a minimax sense. Namely, we consider the MSE of any gradient estimator applied on a collection of twice differentiable functions satisfying a certain bound on their derivative behaviors, but otherwise with unknown function characteristics (including function value and higher-order gradients). Subject to this uncertainty of the function, we consider the minimizer of the worst-case MSE over this function collection, giving rise to what we call the minimax risk. Among the class of linear estimators, we show that, in the one-dimensional case, CFD exactly achieves the minimax risk, whereas in the multi-dimensional case it achieves so up to a multiplicative factor that depends sublinearly on the input dimension. Furthermore, we show that, among the much larger class of all nonlinear estimators, CFD remains nearly minimax optimal up to a multiplicative factor that is polynomial in the dimension.

In terms of methodological contributions, we derive our linear minimax results by using a new elementary proof. We derive our general minimax results via Le Cam's method \cite{tsybakov2009} with the imposition of an adversarially chosen hypothesis test, and the notion of modulus of continuity \cite{donoho1991geometrizing} to obtain the worst-case functions derived from this test. Lastly, we also demonstrate that, without extra knowledge on the gradient, randomized schemes such as simultaneous perturbation will lead to unbounded worst-case risks in general, due to the interaction between the gradient magnitude and the additional variance coming from the random perturbation. This indicates that less conservative analyses along this line would require more information on the magnitude of the gradient of interest.



Our work is related to, and contrasted with, the derivative estimation in nonparametric regression (e.g., \cite{fan1993local,fan1997local}), which focuses on the estimation of the conditional expectations and their derivatives given input values, a similar setting as ours. However, in these studies the data and in particular the available input values are often assumed given a priori. In contrast, in stochastic gradient estimation and optimization, one often has the capability to select the input points at which the function evaluation is conducted. This therefore endows more flexibility than nonparametric regression and, correspondingly, translates to superior minimax rates in our setting. For example, \cite{fan1997local} has established a minimax risk of order $O(n^{-4/7})$ for nonparametric derivative estimation, which is slower than our $O(n^{-2/3})$. Finally, we note that other works \cite{dai2016optimal,de2013derivative,wang2015derivative} have studied derivative estimation uniformly well over regular or equi-spaced input design points. Moreover, these papers consider asymptotic risks as the number of input points grow, in contrast to the finite-sample results in this work.


The remainder of the paper is as follows. Section \ref{s1} focuses on linear minimax risk and the corresponding optimal or nearly optimal estimators. Section \ref{sec:general} focuses on general risks and estimators. Section \ref{sec:conclusion} concludes the paper and discusses future directions.

\section{Linear Minimax Risk and Optimal Estimators}
\label{s1}
In this section we focus on the class of linear estimators. Section \ref{sec:single} first presents the single-dimensional case. Section \ref{s2} generalizes it to the higher-dimensional counterpart, and Section \ref{sec:randomized} further studies the use of simultaneous random perturbation in this setting. We will derive bounds on minimax risks and show that CFD is a nearly optimal estimator. Furthermore, these bounds are tight for any finite sample in the single-dimensional case, and also in the multi-dimensional case under additional restrictions.

\subsection{Single-Dimensional Case}\label{sec:single}
We first introduce our setting. Let $f(\cdot):\mathbb{R}\to\mathbb{R}$ be a performance measure of interest, where we have access to an unbiased estimate $Y(x)$ for any chosen $x\in\mathbb{R}$. In other words, $Y(\cdot)$ is a family of random variables indexed by $x$ such that $E[Y(x)]=f(x)$ and $Var(Y(x))=\sigma^{2}(x)$ for any $x\in\mathbb{R}$. Suppose $x_{0}$ is the point of interest. Our goal is to estimate the derivative $f^{'}(x_{0})$.

Given simulation budget $n\geq1$, we can simulate independently at input design points $x_{0}+\delta_{j},j=1,\cdots,n$, with $\delta_j$ of our choice, giving outputs $Y_j(x_0+\delta_j)$'s. We consider the class of linear estimators in the form
\begin{equation}
L_{n}=\sum_{j=1}^{n}w_{j}Y_j(x_{0}+\delta_{j}),\label{linear}
\end{equation}
where $w_{j}$ are the linear coefficients or weights. Note that for even budget $n$ the CFD scheme
\begin{equation}
\bar L_n=\frac{1}{n/2}\sum_{j=1}^{n/2}\frac{Y_{2j-1}(x_0+\delta)-Y_{2j}(x_0-\delta)}{2\delta}\label{CFD}
\end{equation}
is an example of linear estimators where $\delta_{j}=(-1)^{j+1}\delta$ and $w_{j}= \frac{(-1)^{j+1}}{n\delta}$, for a perturbation size $\delta$.

We aim to study the optimality within the class of all linear estimators in the form \eqref{linear}, and in particular investigate the role of CFD. We use the MSE as a performance criterion, which depends on the a priori unknown function $f$. Our premise is a minimax framework that seeks for schemes to minimize the worst-case MSE, among a class of function $f$ and simulation noise. More precisely, we consider
\begin{eqnarray}
\mathcal{A}&=&\Bigg\{f(\cdot):f^{(2)}(x_{0})\textrm{ exists and }{}\notag\\
&&{}\left|f(x)-f(x_{0})-f^{'}(x_{0})(x-x_{0})-\frac{f^{(2)}(x_{0})}{2}(x-x_{0})^{2}\right|\leq\frac{a}{6}|x-x_{0}|^{3}\Bigg\}\label{class}
\end{eqnarray}
and
\begin{equation}\label{classbsingle}
    \mathcal{B}=\{\sigma^{2}(\cdot):\sigma^{2}(x)\leq b\},
\end{equation}
where $a,b>0$ are assumed given.

Roughly speaking, the objective function class $\mathcal A$ contains all twice differentiable functions whose second-order Taylor-series expansion has a remainder with coefficient (i.e., the third-order derivative if exists) absolutely bounded by $a$. This characterization is not exact, however, since the Taylor-series expansion in \eqref{class} applies only to the point $x_0$, and thus $\mathcal A$ is more general than the aforementioned characterization. The reason for proposing this class, instead of other similar ones, is that $\mathcal A$ allows us to obtain a very accurate minimax analysis and derivation of optimal estimators. On the other hand, the variance function class $\mathcal B$ stipulates that the noise at any point has a variance with upper bound $b$.

We define the linear minimax $L_{2}$-risk as \begin{equation}\label{l2risklinear}
    R(n,\mathcal{A},\mathcal{B})=\inf_{\substack{\delta_{j},j=1,\cdots,n\\w_{j},j=1,\cdots,n}}\sup_{\substack{f(\cdot)\in\mathcal{A}\\\sigma^{2}(\cdot)\in\mathcal{B}}}E\left(L_{n}-f^{'}(x_{0})\right)^{2},
\end{equation}
which is the minimum worst-case MSE among all functions $f\in\mathcal A$ and noise levels in $\mathcal B$. The linear estimators are selected through the design points $x_0+\delta_j$'s and linear coefficients $w_j$'s.

The following theorem gives the exact expression for the minimax risk, and shows that a suitably calibrated CFD attains this risk level. In other words, CFD is the optimal linear estimator among the problem class specified by $(\mathcal A,\mathcal B)$. The proof of this result involves only elementary inequalities.

\begin{theorem}\label{st}
For any $n\geq1$, let the objective function class $\mathcal{A}$ be defined by \eqref{class} and the variance function class $\mathcal{B}$ be defined by \eqref{classbsingle}. Then the linear minimax $L_2$-risk defined by \eqref{l2risklinear} satisfies
$$R(n,\mathcal{A},\mathcal{B})\geq\left(\frac{3a^{2}b^{2}}{16}\right)^{1/3}n^{-2/3}.$$
Moreover, if the budget $n$ is even, the CFD estimator $\bar{L}_{n}$ in \eqref{CFD} with $\delta=(\frac{18b}{a^{2}})^{1/6}\frac{1}{n^{1/6}}$ satisfies
$$\sup_{\substack{f(\cdot)\in\mathcal{A}\\\sigma^{2}(\cdot)\in\mathcal{B}}}E(\bar{L}_{n}-f^{'}(x_{0}))^{2}=\left(\frac{3a^{2}b^{2}}{16}\right)^{1/3}n^{-2/3},$$
and is the best linear estimator in the class of problems defined by $(\mathcal{A},\mathcal{B})$.\hfill\qed
\end{theorem}

\begin{proof}
Consider any designs $\delta_{j},j=1,\cdots,n$, any linear coefficients $w_{j},j=1,\cdots,n$, and any function $f(\cdot)\in \mathcal{A}$. If  $f(\cdot)\in C^{3}(\mathbb{R})$, we have, by Taylor's expansion,
$$f(x_{0}+\delta_{j})=f(x_{0})+f^{'}(x_{0})\delta_{j}+\frac{f^{(2)}(x_{0})}{2}\delta_{j}^{2}+\frac{f^{(3)}(x_{0}+t_{j}\delta_{j})}{6}\delta_{j}^{3},$$
for any $j=1,\cdots,n$, where $0\leq t_{j}\leq 1$. Thus under the additional $C^3$ assumption (which we will verify if necessary) the bias of the estimator $L_{n}$ is
\begin{equation}\label{biaslinearc3}
    EL_{n}-f^{'}(x_{0})=f(x_{0})\sum_{j=1}^{n}w_{j}+f^{'}(x_{0})\left(\sum_{j=1}^{n}w_{j}\delta_{j}-1\right)+\frac{f^{(2)}(x_{0})}{2}\sum_{j=1}^{n}w_{j}\delta_{j}^{2}+\sum_{j=1}^{n}\frac{f^{(3)}(x_{0}+t_{j}\delta_{j})}{6}w_{j}\delta_{j}^{3}.
\end{equation}
On the other hand, the variance of the estimator $L_{n}$ is
$$Var(L_{n})=\sum_{j=1}^{n}w_{j}^{2}\sigma^{2}(x_{0}+\delta_{j}),$$
which holds regardless of whether $f(\cdot)\in C^3(\mathbb R)$. If $\sum_{j=1}^{n}w_{j}\neq0$, we consider the particular case where $f(\cdot)$ is a constant function, i.e., $f(x)=f(x_0)$ for all $x\in\mathbb{R}$. Note that constant functions are in $\mathcal{A}$ and also in $C^3(\mathbb{R})$. Also note that $f^{'}(x)=f^{(2)}(x)=f^{(3)}(x)=0$ for all $x$, therefore from \eqref{biaslinearc3} we conclude that
$$\sup_{f(\cdot)\in \mathcal{A}}(EL_{n}-f^{'}(x_{0}))^{2}\geq\sup_{f(x_0)\in\mathbb{R}}f(x_0)^2(\sum_{j=1}^nw_j)^2=\infty,$$
and hence $\sup_{f(\cdot)\in \mathcal{A}}E(L_{n}-f^{'}(x_{0}))^{2}=\sup_{f(\cdot)\in \mathcal{A}}(EL_{n}-f^{'}(x_{0}))^{2}+Var(L_n)=\infty$. Therefore, for the purpose of deriving a lower bound for $R(n,\mathcal{A},\mathcal{B})$, we can assume without loss of generality that $\sum_{j=1}^{n}w_{j}=0$. Similarly we can assume $\sum_{j=1}^{n}w_{j}\delta_{j}-1=0$ (by considering functions $f(\cdot)\in\mathcal{A}$ that are linear, and using the expression \eqref{biaslinearc3}). Furthermore, if $\delta_{i}=\delta_{j}$, we assume w.l.o.g. $w_{i}=w_{j}$ since it leads to smaller $Var(L_n)$. Now consider $f(\cdot)\in\mathcal{A}$ such that $f(x_{0}+\delta_{j})=\frac{a}{6}|\delta_{j}|^{3}\cdot sign(w_{j})$, and $f(x)=0$ otherwise. In such a case the MSE simplifies to
$$\left(\sum_{j=1}^{n}\frac{a}{6}|w_{j}\delta_{j}^{3}|\right)^{2}+\sum_{j=1}^{n}w_{j}^{2}\sigma^{2}(x_{0}+\delta_{j}).$$
Further considering the case $\sigma^{2}(x_{0}+\delta_{j})=b$, and using the bias-variance decomposition of MSE, we get a lower bound
\begin{equation}\label{lowboundlinear}
    \sup_{\substack{f(\cdot)\in\mathcal{A}\\\sigma^{2}(\cdot)\in\mathcal{B}}}E(L_{n}-f^{'}(x_{0}))^{2}\geq\frac{a^{2}}{36}\left(\sum_{j=1}^{n}|w_{j}\delta_{j}^{3}|\right)^{2}+b\sum_{j=1}^{n}w_{j}^{2}.
\end{equation}
Now since $\sum_{j=1}^{n}w_{j}\delta_{j}=1$, by H\"older's inequality,
$$1\leq\sum_{j=1}^{n}|w_{j}|^{1/3}|\delta_{j}||w_{j}|^{2/3}\leq\left(\sum_{j=1}^{n}|w_{j}\delta_{j}^{3}|\right)^{1/3}\left(\sum_{j=1}^{n}|w_{j}|\right)^{2/3}$$
and so
\begin{equation}\label{lowboundpart}
\left(\sum_{j=1}^{n}|w_{j}\delta_{j}^{3}|\right)^{2}\geq\frac{1}{\left(\sum_{j=1}^{n}|w_{j}|\right)^{4}}\geq\frac{1}{n^{2}\left(\sum_{j=1}^{n}w_{j}^{2}\right)^{2}},
\end{equation}
where we use the Cauchy-Schwarz inequality $\left(\sum_{j=1}^n|w_j|\right)^2\leq n\left(\sum_{j=1}^nw_j^2\right)$. Thus, combining \eqref{lowboundlinear} and \eqref{lowboundpart},
\begin{equation}\label{lowboundafterholder}
    \sup_{\substack{f(\cdot)\in\mathcal{A}\\\sigma^{2}(\cdot)\in\mathcal{B}}}E\left(L_{n}-f^{'}(x_{0})\right)^{2}\geq\frac{a^{2}}{36}\frac{1}{n^{2}\left(\sum_{j=1}^{n}w_{j}^{2}\right)^{2}}+b\sum_{j=1}^{n}w_{j}^{2}.
\end{equation}
Treat the right hand side of \eqref{lowboundafterholder} as a function of $\sum_{j=1}^nw_j^2$. Note that by the first-order optimality condition, we have
\begin{equation}\label{minlowbound}
\min_{H>0}\frac{a^{2}}{36}\frac{1}{n^{2}H^{2}}+b H=\left(\frac{3a^{2}b^{2}}{16}\right)^{1/3}n^{-2/3},
\end{equation}
where the minimum is achieved at
$$H=\left(\frac{a^{2}}{b}\right)^{1/3}\frac{1}{18^{1/3}n^{2/3}}.$$
Since $\delta_{j},w_{j}$ are arbitrary, from \eqref{lowboundafterholder} and \eqref{minlowbound} we conclude that
\begin{equation}\label{lb}
\inf_{\substack{\delta_{j},j=1,\cdots,n\\w_{j},j=1,\cdots,n}}\sup_{\substack{f(\cdot)\in\mathcal{A}\\\sigma^{2}(\cdot)\in\mathcal{B}}}E\left(L_{n}-f^{'}(x_{0})\right)^{2}\geq\left(\frac{3a^{2}b^{2}}{16}\right)^{1/3}n^{-2/3}.
\end{equation}
On the other hand, supposing the budget $n$ is even, we see that the bias of the estimator $\bar{L}_{n}$ satisfies
\begin{eqnarray*}
&&|E\bar{L}_{n}-f^{'}(x_{0})|=\left|\frac{f(x_{0}+\delta)-f(x_{0}-\delta)}{2\delta}-f^{'}(x_{0})\right|=\left|\frac{f(x_{0}+\delta)-f(x_{0}-\delta)-2\delta f^{'}(x_{0})}{2\delta}\right|\\
&&=\frac{\left|\left(f(x_{0}+\delta)-f(x_{0})-f^{'}(x_{0})\delta-\frac{f^{(2)}(x_{0})}{2}\delta^{2}\right)-\left(f(x_{0}-\delta)-f(x_{0})+f^{'}(x_{0})\delta-\frac{f^{(2)}(x_{0})}{2}\delta^{2}\right)\right|}{2\delta}\\
&&\leq \frac{\left|f(x_{0}+\delta)-f(x_{0})-f^{'}(x_{0})\delta-\frac{f^{(2)}(x_{0})}{2}\delta^{2}\right|+\left|f(x_{0}-\delta)-f(x_{0})+f^{'}(x_{0})\delta-\frac{f^{(2)}(x_{0})}{2}\delta^{2} \right|}{2\delta}\\
&&\leq \frac{a}{6}\delta^{2}.
\end{eqnarray*}
Also, the variance satisfies $Var(\bar{L}_{n})\leq \frac{b}{n\delta^{2}}$. Thus the MSE satisfies
\begin{equation}\label{ub}
E\left(\bar{L}_{n}-f^{'}(x_{0})\right)^{2}\leq\frac{a^{2}}{36}\delta^{4}+\frac{b}{n\delta^{2}}=\left(\frac{3a^{2}b^{2}}{16}\right)^{1/3}n^{-2/3}
\end{equation}
by plugging in $\delta=(\frac{18b}{a^{2}})^{1/6}\frac{1}{n^{1/6}}$. We note that the bound (\ref{ub}) holds for any $f(\cdot)\in\mathcal{A}$ and $\sigma^{2}(\cdot)\in\mathcal{B}$. Combining (\ref{lb}) and (\ref{ub}), we conclude the proof for Theorem~\ref{st}.
\end{proof}

\subsection{Multi-Dimensional Case}
\label{s2}
We now generalize to the multi-dimensional case. Consider a performance measure with multi-dimensional input, $f(\cdot):\mathbb{R}^{p}\to\mathbb{R}$ where $p\geq2$. Analogous to the single-dimensional case, suppose $Y(\cdot)$ is an unbiased estimate of $f(\cdot)$. We would like to estimate $\nabla f(x_{0})$ where $x_{0}\in\mathbb{R}^{p}$ is the point of interest. Consider
\begin{eqnarray}
\mathcal{A}_{q}&=&\Bigg\{f(\cdot):\nabla^{2} f(x_{0})\textrm{ exists and }{}\notag\\
&&{}\left|f(x)-f(x_{0})-\nabla f(x_{0})^{T}(x-x_{0})-\frac{1}{2}(x-x_{0})^{T}\nabla^{2}f(x_{0})(x-x_{0})\right|\leq \frac{a}{6}\|x-x_{0}\|^{3}_{q}\Bigg\}\label{classmulti}
\end{eqnarray}
and
\begin{equation}\label{classbmulti}
\mathcal{B}=\{\sigma^{2}(\cdot):\sigma^{2}(x)\leq b\},
\end{equation}
where $a,b>0$, and $\|\cdot\|_q$ denotes the $\ell_q$-norm, $q\in\{1,2,\infty\}$.
Given simulation budget $n\geq1$, we simulate independently at design points $x_{0}+\delta_{j},j=1,\cdots,n$, and form the vector-valued linear estimator
$$L^{p}_{n}=\sum_{j=1}^{n}w_{j}Y_j(x_{0}+\delta_{j}),$$
where $w_{j}\in\mathbb R^p$ are the vector-valued linear coefficients. Correspondingly, suppose $n$ is a multiple of $2p$, and we divide the budget into $n/p$ for each dimension $i$ where we compute
\begin{equation}
(\bar L_{n}^{p})_{i}=\frac{1}{n/(2p)}\sum_{j=1}^{n/(2p)}\frac{Y_{2j-1}(x_0+\delta e_i)-Y_{2j}(x_0-\delta e_i)}{2\delta}\label{CFD multi}
\end{equation}
with $e_i$ denoting the $i$th standard basis in $\mathbb R^p$. This leads to  $\bar L_n^p=[(\bar L_{n}^p)_{i}]_{i=1,\ldots,p}$ which is the CFD estimator in $\mathbb R^p$ that is the multivariate analog to \eqref{CFD}.

Like before, we define the linear minimax $L_{2}$-risk as \begin{equation}\label{riskmulti}
R_{p}(n,\mathcal{A}_{q},\mathcal{B})=\inf_{\substack{\delta_{j},j=1,\cdots,n\\w_{j},j=1,\cdots,n}}\sup_{\substack{f(\cdot)\in\mathcal{A}_q\\\sigma^{2}(\cdot)\in\mathcal{B}}}E\|L^{p}_{n}-\nabla f(x_{0})\|^{2}_{2}.
\end{equation}
Intuitively, the main distinction of the multi-dimensional setting described above,  compared to the single-dimensional case, is the restriction on the remainder of the second-order Taylor-series expansion characterized by the $\ell_q$-norm, where the choice of $q$ can affect the resulting risk bounds. The following theorem provides a lower estimate of $R_p$ and shows that applying CFD on each of the $p$ dimensions matches this lower estimate up to a multiplicative factor depending sublinearly on $p$. This implies in particular that multi-dimensional CFD is rate-optimal in the sample size $n$. Moreover, if the signs of the weights $w_j$ in $L_n^p$ are further restricted to be the same across components, then CFD achieves the exact minimax risk when $q=1$ or $2$.

\begin{theorem}\label{mt}
For any $n\geq1$, let the objective function class $\mathcal{A}_q$ be defined by \eqref{classmulti} and the variance function class $\mathcal{B}$ be defined by \eqref{classbmulti}. Then:
\begin{enumerate}
\item The linear minimax $L_2$-risk defined by \eqref{riskmulti} satisfies
\begin{equation}\label{mlbt}
R_{p}(n,\mathcal{A}_{q},\mathcal{B})\geq p^{4/3}\left(\frac{3a^{2}b^{2}}{16}\right)^{1/3}n^{-2/3},\text{\ for\ } q=1,2;
\end{equation}
\begin{equation}\label{mlbt2}
R_{p}(n,\mathcal{A}_{q},\mathcal{B})\geq p^{2/3}\left(\frac{3a^{2}b^{2}}{16}\right)^{1/3}n^{-2/3},\text{\ for\ } q=\infty.
\end{equation}
\item If the budget $n$ is a multiple of $2p$, we allocate $n/p$ budget and form the CFD estimator $\bar L_n^p$ in \eqref{CFD multi} with $\delta=(\frac{18b}{a^{2}})^{1/6}\frac{1}{(n/p)^{1/6}}$ on each dimension. Then we have
\begin{equation}\label{me}
\sup_{\substack{f(\cdot)\in\mathcal{A}_q\\\sigma^{2}(\cdot)\in\mathcal{B}}}E\|\bar{L}^{p}_{n}-\nabla f(x_{0})\|^{2}_{2}\leq  p^{5/3}\left(\frac{3a^{2}b^{2}}{16}\right)^{1/3}n^{-2/3},
\end{equation}
and the estimator $\bar{L}_{n}^{p}$ is optimal in the class of problems defined by $(\mathcal{A}_{q},\mathcal{B})$ up to a sublinear multiplicative factor in $p$.
\item If in $L_n^p$ we further restrict each coefficient $w_{j}$ to have the same sign across components, i.e. $sign((w_{j})_{k})=sign((w_{j})_{l})$ for any $1\leq k,l\leq p$, then we have
$$R_{p}(n,\mathcal{A}_{q},\mathcal{B})\geq p^{5/3}\left(\frac{3a^{2}b^{2}}{16}\right)^{1/3}n^{-2/3},\text{\ for\ } q=1,2;$$
$$R_{p}(n,\mathcal{A}_{q},\mathcal{B})\geq p\left(\frac{3a^{2}b^{2}}{16}\right)^{1/3}n^{-2/3},\text{\ for\ } q=\infty$$
for this restricted class of linear estimators, so that $\bar L_n^p$ is exactly optimal when considering function class $\mathcal{A}_{q}$ with $q=1,2$.\end{enumerate}\hfill\qed
\end{theorem}

Theorem \ref{mt} is proved in a similar manner as the single-dimensional case in Theorem \ref{st}, but with some specific differences due to different choices of the norm function. The detailed proof of Theorem \ref{mt} is provided in Appendix \ref{missing}. We remark that the $p^{1/3}$ gap between~(\ref{mlbt}) and~(\ref{me}) is due to a technical challenge that the $l_{\infty}\to l_{2}$ matrix norm does not admit an explicit expression. This challenge is bypassed if one restricts the same sign across all components in each coefficient, which recovers the minimax optimality of CFD.

\subsection{Randomized Design}\label{sec:randomized}
The discussion above has focused on deterministic designs in which the perturbation sizes $\delta_j$ are fixed. In multi-dimensional gradient estimation, it is also common to use random perturbation in which a random vector in $\mathbb R^p$ is generated and FD is taken simultaneously for all dimensions by projecting the vector onto each direction. This leads to schemes such as simultaneous perturbation  \cite{spall1992multivariate} and Gaussian smoothing \cite{nesterov2017random}, and are frequently used in stochastic optimization. A question arises how these randomized schemes perform with respect to our presented risk criterion.


To proceed, let $\delta=(\delta^i)_{i=1,\ldots,p}$ be a random vector in $\mathbb{R}^{p}$ where $\delta^{i},i=1,\cdots,p$ are i.i.d. symmetrically distributed about $0$ and satisfy some additional properties (described in, e.g., \cite{spall1992multivariate}, which will not concern us as we will see), and let $\phi(\delta)=(1/\delta^{1},\cdots,1/\delta^{p})^{T}$. Other distributional choices of $\delta$ and the associated $\phi$ have also  been suggested (e.g., \cite{nesterov2017random,flaxman2005online}), for which our subsequent argument follows similarly. Suppose the simulation budget $n$ is even. We choose a scaling parameter $h>0$, then repeatedly and independently simulate $\delta_j\in\mathbb R^p$ and $Y_j(\cdot)$'s and output
\begin{equation}\label{simulperturb}
S_n=\frac{1}{n/2}\sum_{j=1}^{n/2}\frac{Y_{2j-1}(x_{0}+h\delta_{j})-Y_{2j}(x_{0}-h\delta_{j})}{2h}\phi(\delta_j).
\end{equation}
The following theorem shows that, without further assumptions on the magnitude of the first-order gradient, the $L_{2}$-risk of random perturbation can be arbitrarily large.
\begin{theorem}
For any $n\geq2$ even, let the objective function class $\mathcal{A}_q$ be defined by \eqref{classmulti} and the variance function class $\mathcal{B}$ be defined by \eqref{classbmulti}. Then the estimator defined by \eqref{simulperturb} satisfies
$$\sup_{\substack{f(\cdot)\in\mathcal{A}_q\\\sigma^{2}(\cdot)\in\mathcal{B}}}E\|S_{n}-\nabla f(x_{0})\|^{2}_{2}=\infty.$$\hfill\qed\label{SP thm}
\end{theorem}

\begin{proof}
First note that by independence,
{\small
$$E\|S_{n}-ES_{n}\|^{2}_{2}=\frac{2}{n}tr\left(Cov\left(\frac{Y(x_{0}+h\delta)-Y(x_{0}-h\delta)}{2h}\phi(\delta)\right)\right)=\frac{2}{n}\sum_{i=1}^{p}Var\left(\frac{Y(x_{0}+h\delta)-Y(x_{0}-h\delta)}{2h\delta^{i}}\right).$$}
Next, by conditioning on $\delta$, we have
{\small
\begin{align*}
Var\left(\frac{Y(x_{0}+h\delta)-Y(x_{0}-h\delta)}{2h\delta^{i}}\right)&=Var\left(\frac{f(x_{0}+h\delta)-f(x_{0}-h\delta)}{2h\delta^{i}}\right)+E\left[\frac{\sigma^{2}(x_{0}+h\delta)+\sigma^{2}(x_{0}-h\delta)}{4h^{2}(\delta^{i})^{2}}\right]\\
&\geq Var\left(\frac{f(x_{0}+h\delta)-f(x_{0}-h\delta)}{2h\delta^{i}}\right).
\end{align*}}

Now consider the function $f(x)=f(x_{0})+\rho \mathbbm{1}^{T}(x-x_{0})$ in $\mathcal{A}_{q}$, where $\rho\in\mathbb{R}$ and $\mathbbm{1}\in\mathbb{R}^{p}$ denotes the vector of all ones. Thus we get
$$Var\left(\frac{f(x_{0}+h\delta)-f(x_{0}-h\delta)}{2h\delta^{i}}\right)=\rho^{2}Var\left(\frac{\mathbbm{1}^{T}\delta}{\delta^{i}}\right)>0.$$
Sending $\rho\to\infty$ we conclude that
$$\sup_{f(\cdot)\in\mathcal{A}_{q}}Var\left(\frac{f(x_{0}+h\delta)-f(x_{0}-h\delta)}{2h\delta^{i}}\right)=\infty.$$
Finally note the bias-variance decomposition
{\small
$$E\|S_{n}-\nabla f(x_{0})\|^{2}_{2}=\|ES_{n}-\nabla  f(x_{0})\|^{2}_{2}+E\|S_{n}-ES_{n}\|_{2}^{2}\geq E\|S_{n}-ES_{n}\|_{2}^{2}\geq\frac{2}{n}\sum_{i=1}^{p}Var\left(\frac{f(x_{0}+h\delta)-f(x_{0}-h\delta)}{2h\delta^{i}}\right),$$
}which completes our proof.
\end{proof}

The unboundedness of the worst-case $L_2$-risk in Theorem \ref{SP thm} is due to the interaction between the gradient of interest and the variance from the random perturbation. This hints that in general, to restrain the worst-case $L_2$-risk for such schemes, extra knowledge on the magnitude of the gradient is needed.

\section{General Minimax Risk}\label{sec:general}
We now expand our analysis to consider estimators that are possibly nonlinear. Section \ref{sec:single general} first presents the single-dimensional case. Section \ref{sec:multi general} then presents the generalization to the multi-dimensional counterpart. We derive bounds for the minimax risks and show that, in these expanded classes, the CFD estimators are still nearly optimal.

\subsection{Single-Dimensional Case}\label{sec:single general}
Adopting the notations from the previous sections, suppose the budget is $n$. We select the input design points $x_{1},\cdots,x_{n}$ and for convenience let $Y_{j}=Y_j(x_{j})$ be the independent unbiased noisy function evaluation of $f$ at $x_{j}$ with simulation variance $\sigma^{2}(x_{j})$. We are interested in estimating $f^{'}$ at $x_0$. Denote $\widehat{\theta}=\widehat{\theta}(Y_{1},\cdots,Y_{n})$ as a generic estimator. Like in Section \ref{sec:single}, we consider the class of problems specified by the objective function class $\mathcal A$ defined in \eqref{class} and the variance function class $\mathcal B$ defined in \eqref{classbsingle}.
Now we define the minimax $L_2$-risk as
\begin{equation}\label{generalrisk}
R(n,\mathcal{A},\mathcal{B})=\inf_{\substack{x_{j},j=1,\cdots,n\\\widehat{\theta}}}\sup_{\substack{f(\cdot)\in\mathcal{A}\\\sigma^{2}(\cdot)\in\mathcal{B}}}E(\widehat{\theta}-f^{'}(x_0))^{2}.
\end{equation}
Note that here there is no restriction on how the estimator $\widehat\theta$ depends on $Y_1,\ldots,Y_n$. 
\begin{theorem}
For any $n\geq1$, let the objective function class $\mathcal{A}$ be defined by \eqref{class} and the variance function class $\mathcal{B}$ be defined by \eqref{classbsingle}. Then the minimax $L_2$-risk defined by \eqref{generalrisk} satisfies
\begin{equation}
R(n,\mathcal{A},\mathcal{B})\geq \frac{1}{16}e^{-2/3}\left(\frac{3ab}{n}\right)^{2/3}.\label{risk bound}
\end{equation}
Consequently, the CFD estimator $\bar L_n$ in \eqref{CFD} is optimal up to a constant multiplicative factor.\hfill\qed\label{general thm}
\end{theorem}

The last conclusion in Theorem \ref{general thm} is a simple consequence from combining the risk estimate of $\bar L_n$ in Theorem \ref{st} with \eqref{risk bound}. In contrast to the elementary proof for Theorem \ref{st}, here we use Le Cam's method (e.g., \cite{tsybakov2009}) and the notion of modulus of continuity  \cite{donoho1991geometrizing} to estimate the minimax risk. Specifically, Le Cam's method derives minimax lower bounds by constructing a hypothesis test and using its error to inform a bound. The error of the hypothesis test, in turn, is analyzable by the Neyman-Pearson lemma. The lower bound provided by Le Cam's method involves the distance between two functions, and tightening the lower bound then becomes a functional optimization problem that can be viewed as the dual or inverse of the formulation to attain the so-called modulus of continuity. Consequently, finding the extremal or worst-case functions for the inverse modulus of continuity will give the resulting lower bound.

\begin{proof}
For convenience, we set $x_0=0$ w.l.o.g. in this proof, so that
\begin{equation}\label{classgeneral}
\mathcal{A}=\left\{f(\cdot):f^{(2)}(0)\textrm{ exists and } \left|f(x)-f(0)-f^{'}(0)x-\frac{f^{(2)}(0)}{2}x^{2}\right|\leq\frac{a}{6}|x|^{3}\right\}
\end{equation}
We use Le Cam's method. Consider arbitrary functions $f_{1},f_{2}\in\mathcal{A}$. For any estimator $\widehat{\theta}$, we have
\begin{equation}\label{generalineq1}
(f_{1}^{'}(0)-f_{2}^{'}(0))^{2}\leq 2 (\widehat{\theta}-f_{1}^{'}(0))^{2}+2(\widehat{\theta}-f_{2}^{'}(0))^{2}.
\end{equation}
For convenience denote the quantity $|f^{'}_1(0)-f^{'}_2(0)|$ as $\epsilon$.
Define test statistic $\psi$ by
\begin{equation}\label{generalteststat}
\psi(Y_{1},\cdots,Y_{n}) = \left\{ \begin{array}{ll}
1 & \textrm{if } |\widehat{\theta}-f_{2}^{'}(0)|\leq|\widehat{\theta}-f_{1}^{'}(0)|,\\
2 & \textrm{if } |\widehat{\theta}-f_{2}^{'}(0)|>|\widehat{\theta}-f_{1}^{'}(0)|.
\end{array} \right.
\end{equation}
Let $E_{k},k=1,2$ denote the expectation (and $P_{k}$ and $p_{k}$ as the probability measure and density) under model
$$Y_{j}\sim f_{k}(x_{j})+\eta_{j},j=1,\cdots,n,$$
where $\eta_{j},j=1,\cdots,n$ i.i.d. follows a normal distribution with mean zero and variance $b$. We have
$$E_{k}(\widehat{\theta}-f_{k}^{'}(0))^{2}\geq E_{k}\left[(\widehat{\theta}-f_{k}^{'}(0))^{2}I(\psi=k)\right]\geq\frac{\epsilon^{2}}{4}P_{k}(\psi=k),$$
for $k=1,2$, where the second inequality comes from \eqref{generalineq1} and \eqref{generalteststat}.
Thus
$$\sup_{\substack{f(\cdot)\in\mathcal{A}\\\sigma^{2}(\cdot)\in\mathcal{B}}}E(\widehat{\theta}-f^{'}(0))^{2}\geq\max_{k=1,2}E_{k}(\widehat{\theta}-f_{k}^{'}(0))^{2}\geq \frac{\epsilon^{2}}{4}\frac{P_{1}(\psi=1)+P_{2}(\psi=2)}{2}.$$
Taking infimum over all possible estimators, we get
$$\inf_{\widehat{\theta}}\sup_{\substack{f(\cdot)\in\mathcal{A}\\\sigma^{2}(\cdot)\in\mathcal{B}}}E(\widehat{\theta}-f^{'}(0))^{2}\geq\frac{\epsilon^{2}}{8}\inf_{\psi}(P_{1}(\psi=1)+P_{2}(\psi=2)).$$
The right hand side is minimized by the Neyman-Pearson test, i.e.,
\begin{displaymath}
\psi_{0}(y_{1},\cdots,y_{n})= \left\{ \begin{array}{ll}
1 & \textrm{if }p_{2}(y_{1},\cdots,y_{n})\geq p_{1}(y_{1},\cdots,y_{n}),\\
2& \textrm{if }p_{2}(y_{1},\cdots,y_{n})<p_{1}(y_{1},\cdots,y_{n}).
\end{array} \right.
\end{displaymath}
Thus by Lemma 2.6 in \cite{tsybakov2009}
(which we include in Appendix~\ref{appendlemma} for self-contained purpose),
\begin{equation}\label{stdrelation}
\inf_{\widehat{\theta}}\sup_{\substack{f(\cdot)\in\mathcal{A}\\\sigma^{2}(\cdot)\in\mathcal{B}}}E(\widehat{\theta}-f^{'}(0))^{2}\geq\frac{\epsilon^{2}}{8}\int \min\{p_{1}(y_{1},\cdots,y_{n}),p_{2}(y_{1},\cdots,y_{n})\}dy\geq\frac{\epsilon^{2}}{16}e^{-KL(P_{1},P_{2})},
\end{equation}
where $KL(P_{1},P_{2})$ denote the KL divergence between the distributions $P_{1},P_{2}$. Now since $P_{1}\sim\mathcal{N}(\mu_{1},bI_{n\times n})$ and $P_{2}\sim\mathcal{N}(\mu_{2},bI_{n\times n})$, where $\mu_{k}=(f_{k}(x_{1}),\cdots,f_{k}(x_{n})),k=1,2$, by direct computation,
\begin{equation}\label{klexpr}
KL(P_{1},P_{2})=\frac{1}{2b}(\mu_{2}-\mu_{1})^{T}(\mu_{2}-\mu_{1})=\frac{1}{2b}\|\mu_{2}-\mu_{1}\|_{2}^{2}\leq\frac{n}{2b}\underset{x}{\sup}|f_{1}(x)-f_{2}(x)|^2.
\end{equation}

To this end, we consider a constrained functional optimization problem for each $\epsilon$:
\begin{equation}
\omega(\epsilon)=\inf\left\{\underset{x}{\sup}|f_{1}(x)-f_{2}(x)|:|f'_{1}(0)-f'_{2}(0)|=\epsilon,f_{1},f_{2}\in\mathcal{A}\right\}.\label{modulus new}
\end{equation}
Solving the optimization in \eqref{modulus new} yields a tightest upper bound for $KL(P_{1},P_{2})$, and further obtains a tightest lower bound for $\frac{\epsilon^{2}}{16}e^{-KL(P_{1},P_{2})}$. The decision variables in the outer problem in \eqref{modulus new} are a pair of functions $(f_{1},f_{2})$, where we would later denote $(f_{1}^{*},f_{2}^{*})$ as the solutions, which constitute the extremal or worst-case functions. Note that $\omega(\epsilon)$ is the inverse function of the so-called modulus of continuity at the point $x_{0}=0$, which is defined by $$\epsilon(\omega)=\sup\left\{|f'_{1}(0)-f'_{2}(0)|:\underset{x}{\sup}|f_{1}(x)-f_{2}(x)|\leq\omega,f_{1},f_{2}\in\mathcal{A}\right\}.$$
By Lemma 7 of~\cite{donoho1991geometrizing} (which we include in Appendix~\ref{appendlemma} for self-contained purpose), the extremal pair of functions in attaining the modulus function $\epsilon(\omega)$ can be chosen in the form $f_{1}^*=f$ and $f_{2}^*=-f$ for some $f$. Thus
\begin{gather}
\omega(\epsilon)=2\inf\left\{\underset{x}{\sup}|f(x)|:|f'(0)|=\epsilon/2,f\in\mathcal{A}\right\}.\label{single_modulus}
\end{gather}
If $f(x)$ solves problem~(\ref{single_modulus}), so does $-f(-x)$. As the absolute value is a convex function, $(f(x)-f(-x))/2$ is then also a solution. Therefore, we can restrict attention to odd functions in our search for a solution to~(\ref{single_modulus}). Note that from the definition of $\epsilon$, we have $|f'(0)|=\epsilon/2$.

For each odd function $f\in\mathcal{A}$,
$$\left|f(x)-f^{'}(0)x\right|=\left|f(x)-f(0)-f^{'}(0)x-\frac{f^{(2)}(0)}{2}x^{2}\right|\leq\frac{a}{6}|x|^{3}.$$
It follows that
$$\left|f(x)\right|\geq \left|f^{'}(0)x\right|-\left|f(x)-f^{'}(0)x\right|\geq\frac{\epsilon}{2}|x|-\frac{a}{6}|x|^{3}.$$
Consider the function $f^{*}$ which increases with a gradient $\epsilon/2$ at $x_{0}=0$ and is as close to 0 as possible:
\begin{gather*}
f^{*}(x)=sign(x)\left[\frac{\epsilon}{2}|x|-\frac{a}{6}|x|^{3}\right]_{+}.
\end{gather*}
It is easy to verify that $f^{*}(x)$ is an odd function, $f^{*}(x)\in\mathcal{A}$, and $\underset{x}{\sup}|f(x)|\geq\underset{x}{\sup}|f^{*}(x)|$ for any odd function $f\in\mathcal{A}$. Therefore, $f^{*}(x)$ is a solution to problem~(\ref{single_modulus}). It is now easy to compute that $\omega(\epsilon)=2\underset{x}{\sup}|f^{*}(x)| =\frac{2\epsilon}{3}\sqrt{\frac{\epsilon}{a}}$.


Since $\omega(\epsilon)=\frac{2\epsilon}{3}\sqrt{\frac{\epsilon}{a}}$, we get from \eqref{stdrelation} and \eqref{klexpr} that
$$\inf_{\widehat{\theta}}\sup_{\substack{f(\cdot)\in\mathcal{A}\\\sigma^{2}(\cdot)\in\mathcal{B}}}E(\widehat{\theta}-f^{'}(0))^{2}\geq\frac{\epsilon^{2}}{16}e^{-\frac{2n\epsilon^{3}}{9ab}}.$$
Now take $\epsilon=\left(\frac{3ab}{n}\right)^{1/3}$ at which $\frac{\epsilon^{2}}{16}e^{-\frac{2n\epsilon^{3}}{9ab}}$ achieves its minimum, we get
$$\inf_{\widehat{\theta}}\sup_{\substack{f(\cdot)\in\mathcal{A}\\\sigma^{2}(\cdot)\in\mathcal{B}}}E(\widehat{\theta}-f^{'}(0))^{2}\geq\frac{1}{16}e^{-2/3}\left(\frac{3ab}{n}\right)^{2/3}.$$
We complete our proof by noting that the above bound holds for any design points $x_{1},\cdots,x_{n}$.
\end{proof}

Figure~\ref{single_wcf} visualizes the above “worst-case” function \begin{gather*}
f^{*}(x)=sign(x)\left[\frac{\epsilon}{2}|x|-\frac{a}{6}|x|^{3}\right]_{+}
\end{gather*}
in function class $\mathcal{A}$. Intuitively, when function evaluations of $f_{1}$ and $f_{2}$ are close but their gradients at $x_{0}=0$ are quite different, we cannot easily find an estimator to minimize the errors in gradient estimation for these two functions at the same time. Such a scenario is considered as the worst case. Consider $(f_{1},f_{2})$ taken in the form $(f,-f)$, and the difference between $f^{'}_{1}(0)$ and $f^{'}_{2}(0)$ is $\epsilon$. As shown in Figure~\ref{single_wcf}, $f^{*}$ increases with a gradient $\epsilon/2$ at $x_{0}=0$. In order to make function evaluations of $f_{1}$ and $f_{2}$ as close as possible, $f^{*}$ needs to be as close to 0 as possible. Ultimately, with these worst-case functions indexed by $\epsilon$, we take $\epsilon=\left(\frac{3ab}{n}\right)^{1/3}$ in our lower-bound derivation to balance the maximization of gradient difference $\epsilon$ and the minimization of function evaluation difference.
\begin{figure}[htbp]
	\centering
	\includegraphics[width=0.5\linewidth]{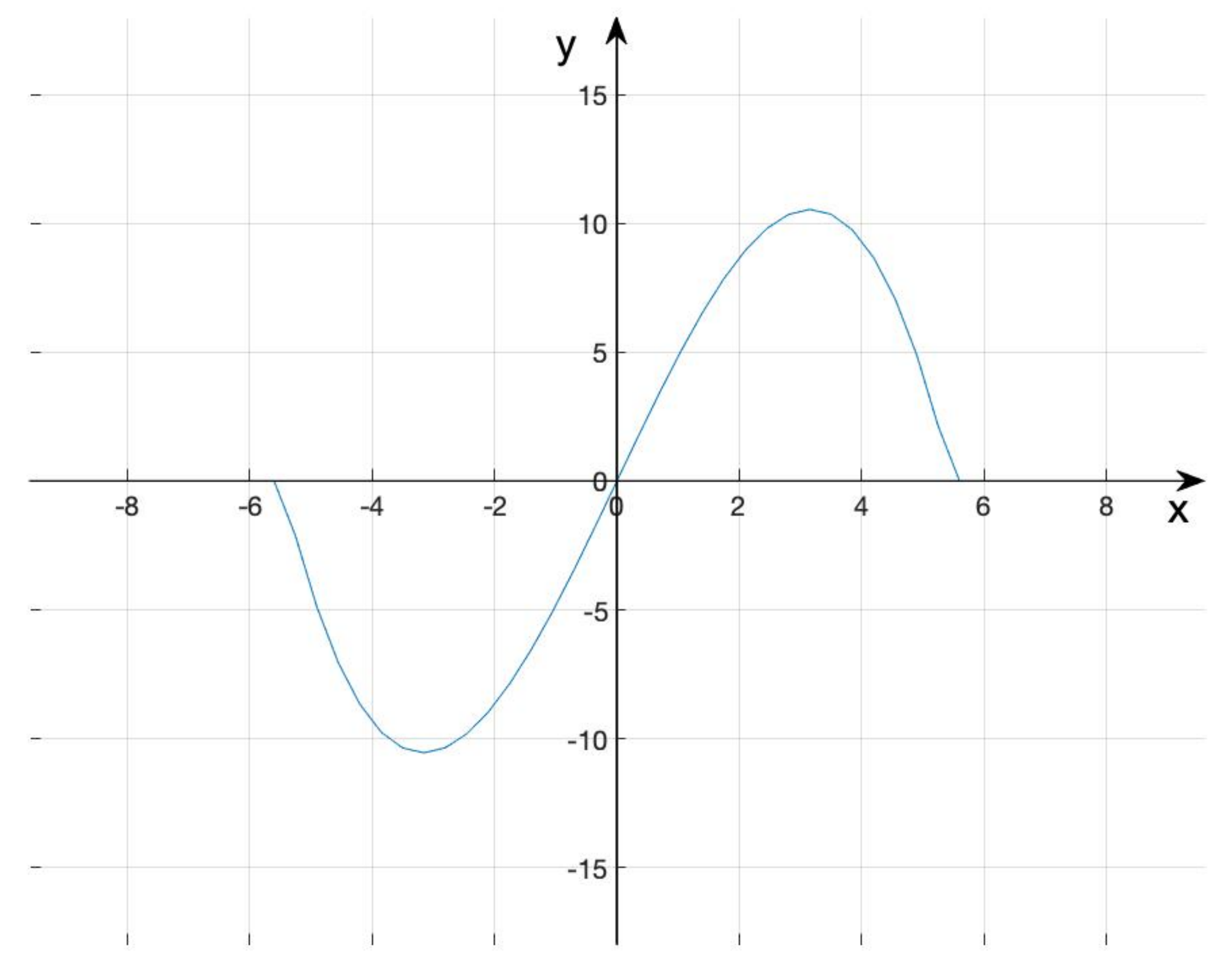}
	\caption{Worst-case function in attaining the inverse modulus of continuity}
    \label{single_wcf}
\end{figure}

 \subsection{Multi-Dimensional Case}\label{sec:multi general}
Suppose now that the input design points are $x_{1},\cdots,x_{n}\in\mathbb{R}^{p}$, and $Y_{j}=Y_j(x_{j})$ are independent unbiased noisy function evaluations of $f:\mathbb{R}^{p}\to\mathbb{R}$ at $x_{j}$ with simulation variance $\sigma^{2}(x_{j})$. We would like to estimate $\nabla f$ at $x_{0}$. Denote $\widehat{\theta}=\widehat{\theta}(Y_{1},\cdots,Y_{n})$ as a generic $\mathbb{R}^{p}$-valued estimator like before. We consider the class of problems specified by the objective function class $\mathcal A_q$ defined in \eqref{classmulti} and the variance function class $\mathcal B$ defined in \eqref{classbmulti},
where $q\in\{1,2,\infty\}$. Define the minimax $L_2$-risk as
\begin{equation}\label{riskgeneralmulti}
R_{p}(n,\mathcal{A}_{q},\mathcal{B})=\underset{\substack{x_{j},j=1,\ldots,n\\\widehat{\theta}}}{\inf}\ \underset{\substack{f(\cdot)\in\mathcal{A}_q\\\sigma^{2}(\cdot)\in\mathcal{B}}}{\sup}\ E\left\|\widehat{\theta}-\nabla f(x_0)\right\|_{2}^{2}.
\end{equation}

\begin{theorem}\label{general multi thm}
For any $n\geq 1$, let the objective function class $\mathcal{A}_q$ be defined by \eqref{classmulti} and the variance function class $\mathcal{B}$ be defined by \eqref{classbmulti}. Then the minimax $L_2$-risk defined by \eqref{riskgeneralmulti} satisfies
\begin{gather*}
R_{p}(n,\mathcal{A}_{q},\mathcal{B})\geq \frac{1}{16}e^{-2/3}\left(\dfrac{3abp^{3/2}}{n}\right)^{2/3}, \text{\ for\ } q=1,\\
R_{p}(n,\mathcal{A}_{q},\mathcal{B})\geq \frac{1}{16}e^{-2/3}\left(\dfrac{3ab}{n}\right)^{2/3}, \text{\ for\ } q=2,\infty.
\end{gather*}
Consequently, the CFD estimator that divides budget equally among all dimensions, $\bar L_n^p$ in~\eqref{CFD multi}, is optimal up to a multiplicative factor polynomial in $p$ for $q=1,2,\infty$.\hfill\qed
\end{theorem}

Similar to the proof of Theorem~\ref{general thm}, Le Cam's method and the modulus of continuity are used to obtain the worst-case hypothesized functions for the general minimax risk in Theorem \ref{general multi thm}. However, here the choice of $q$ in function class $\mathcal{A}_{q}$ affects the modulus function, and thus the resulting worst-case functions and risk bounds. The detailed proof of Theorem \ref{general multi thm} is provided in Appendix \ref{missing}.

Figure~\ref{wcf} visualizes the above worst-case function $f^*$ for a two-dimensional case in function class $\mathcal{A}_{q}$ with different $q$. Similar to the discussion for Figure~\ref{single_wcf}, worst-case functions serve to balance the maximization of the gradient difference and the minimization of the function evaluation difference. Since $\ell_{1}$-norm is the largest among the three considered norms, the worst-case function for $q=1$ descends to 0 most rapidly. The worst-case function for $q=2$ takes a round shape, and the boundary of its zeros is circular. The worst-case function for $q=\infty$ decreases only when the value in the maximal dimension increases and therefore appears the sharpest. In addition, each dimension
of the worst-case function for $q=1$ or $q=2$ has the same derivative at point $x_{0}=0$. However, the worst-case function for $q=\infty$ has a non-zero derivative at point $x_{0}=0$ only along one of its dimensions.
\begin{figure}[htbp]
	\begin{subfigure}{0.33\textwidth}
	  \centering
	  \includegraphics[width=1\linewidth]{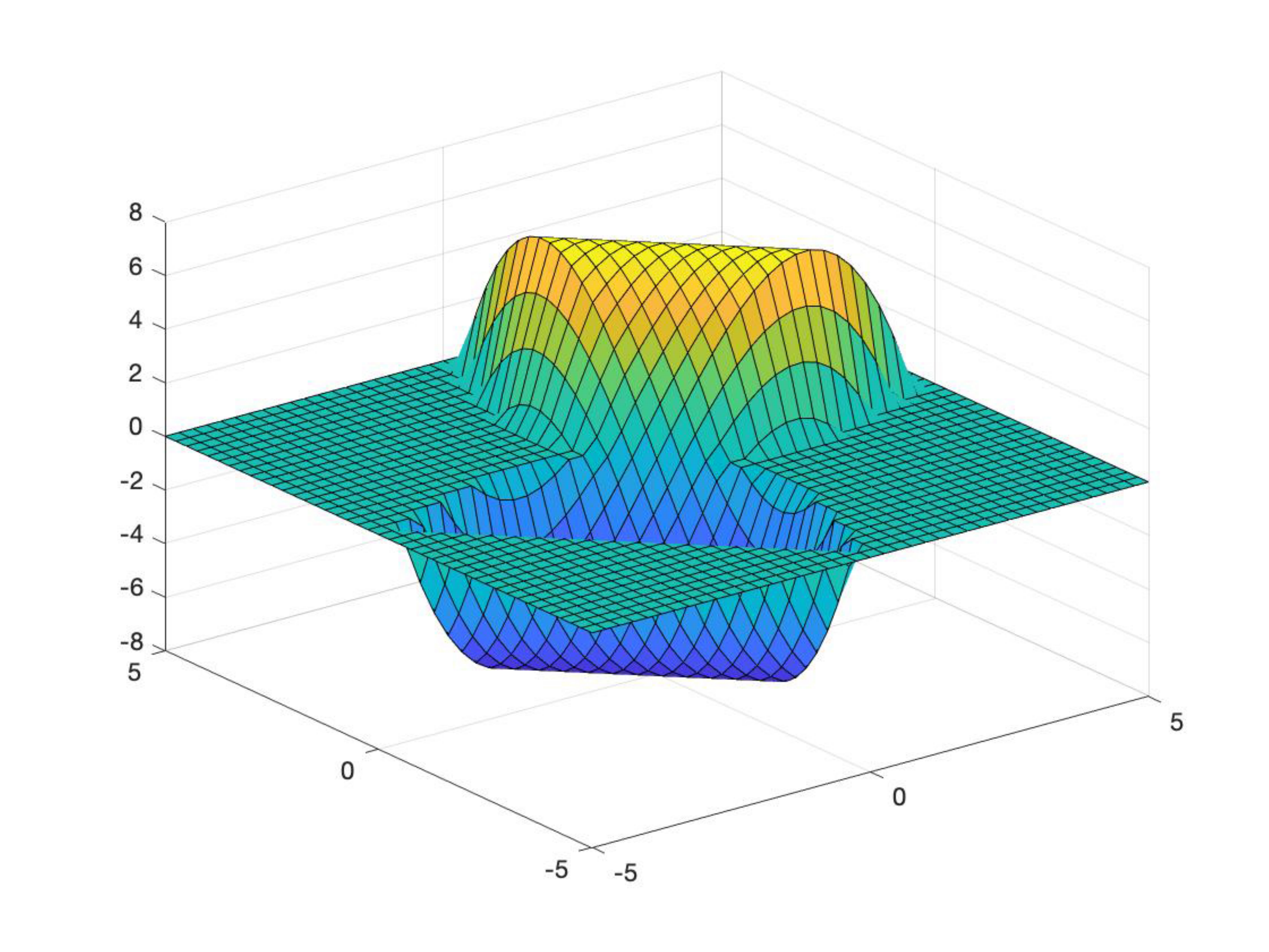}
	  \caption{$q=1$.}\label{function1}
	\end{subfigure}
    \begin{subfigure}{0.33\textwidth}
	  \centering
	  \includegraphics[width=1\linewidth]{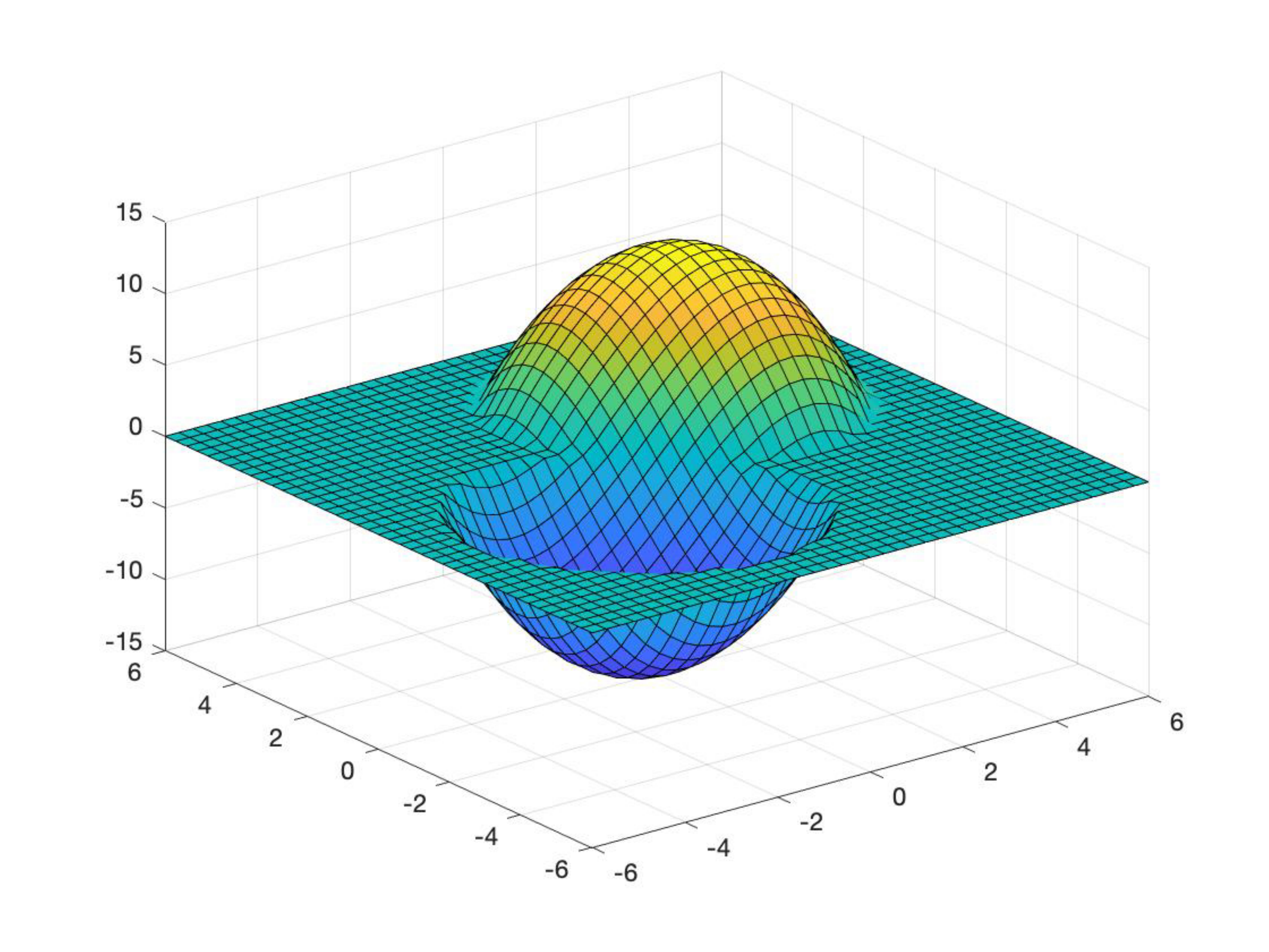}
	  \caption{$q=2$.}\label{function2}
	\end{subfigure}
    \begin{subfigure}{0.33\textwidth}
	  \centering
	  \includegraphics[width=1\linewidth]{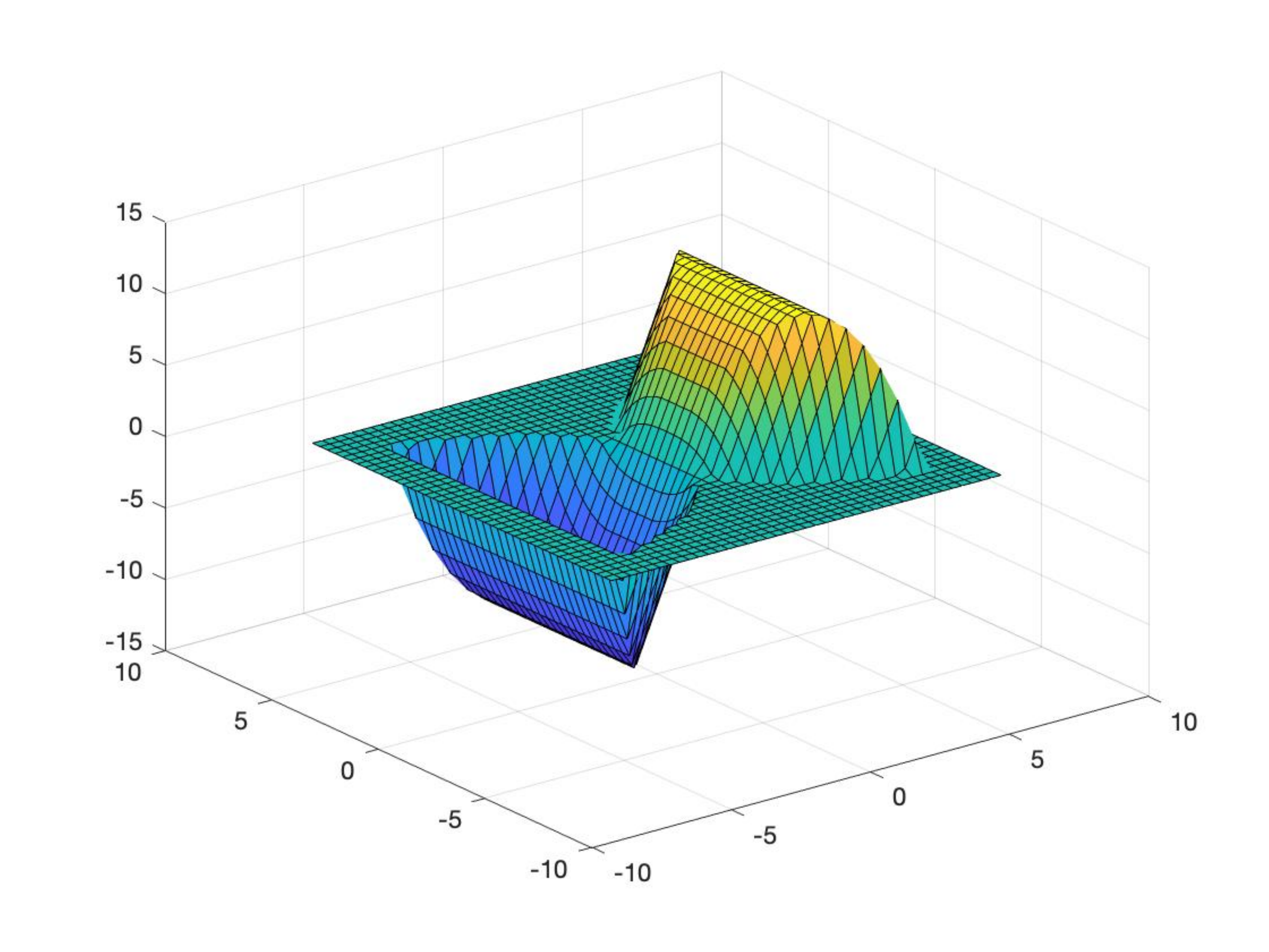}
	  \caption{$q=\infty$.}\label{function3}
	\end{subfigure}
	\caption{Worst-case functions in a two-dimensional case}
    \label{wcf}
\end{figure}

\section{Conclusion}\label{sec:conclusion}
In this paper we studied the minimax optimality of stochastic gradient estimators when only noisy function evaluations are available, with respect to the worst-case MSE among a collection of unknown twice differentiable functions. We derived the exact minimax risk for the class of linear estimators, and showed that CFD is optimal within this class, in the single-dimensional case. We extended the analysis to the multi-dimensional case, by showing the optimality of CFD up to a multiplicative factor sublinear in the dimension, and exactly if the signs of weights in the linear estimator are restricted to be the same across components. We also showed that, without further assumptions on the gradient magnitude, the worst-case risk of random perturbation schemes can be unbounded. Next we approximated the minimax risk over the general class of (nonlinear) estimators and showed that CFD is still nearly optimal over this much larger estimator class. These approximations were shown up to a constant factor in the single-dimensional case and an additional factor depending polynomially on the dimension in the multi-dimensional case. We used elementary techniques in the linear minimax analyses, and Le Cam's method and the modulus of continuity in the general minimax analyses. In future work, we will investigate the use of additional a priori information on the considered function class, and tighten our minimax estimates using potential alternate approaches.

\section*{Acknowledgements}
We gratefully acknowledge support from the National Science Foundation under grants CAREER CMMI-1653339/1834710 and IIS-1849280. A preliminary conference version of this work has appeared in \cite{lam2019minimax}.




\begin{footnotesize}

\bibliographystyle{abbrv}

\bibliography{bbl}
\end{footnotesize}

\newpage

\appendix

\section*{Appendix}
\section{Useful Results}\label{appendlemma}

\begin{lemma}[Lemma 2.6 in
\cite{tsybakov2009}]\label{stdlemma}
For any two distributions $P$ and $Q$ having positive densities $p(\cdot)$ and $q(\cdot)$ on $\mathbb{R}^n$, it holds that
$$\int\min\{p,q\}\geq \frac{1}{2}e^{-KL(P,Q)},$$
where
$$KL(P,Q)=\int p\log \left(\frac{p}{q}\right).$$
\end{lemma}

\begin{lemma}[Lemma 7 in~\cite{donoho1991geometrizing}]\label{pairlemma}
Let functional $T$ be linear and function class $\mathcal{F}$ be convex and centrosymmetric about 0 (that is, $f\in\mathcal{F}$ implies that $-f\in\mathcal{F}$). Define the modules as
\begin{gather}\label{lemma2eq1}
\epsilon(\omega)=\sup\left\{|T(f_{1})-T(f_{2})|:||f_{1}-f_{2}||\leq\omega,f_{1},f_{2}\in\mathcal{F}\right\}.
\end{gather}
Then the modules satisfies
\begin{gather}\label{lemma2eq2}
\epsilon(\omega)=2\sup\left\{|T(f)|:||f||\leq\omega/2,f\in\mathcal{F}\right\}.
\end{gather}
Moreover, if a pair attaining the modules in~\eqref{lemma2eq1} exists, it can be taken to be of the form ($f,-f$) where $f$ is a solution to~\eqref{lemma2eq2}.
\end{lemma}

We note that \cite{donoho1991geometrizing} focuses on the $L_2$-modulus, which uses the $L_2$-norm to define $\|\cdot\|$ in Lemma \ref{pairlemma}. However, the proof of this lemma works for any norm $\|\cdot\|$, including the $L_\infty$-norm needed in our case.

\section{Missing Proofs}\label{missing}
\begin{proof}[Proof of Theorem \ref{mt}]
We prove the three assertions of the theorem one by one.

  \noindent\emph{Proof of Assertion 1.} Consider any designs $\delta_j,j=1,\cdots,n$, any linear coefficients $w_j,j=1,\cdots,n$, and any function $f(\cdot)\in\mathcal{A}_q$. If $f(\cdot)\in C^{3}(\mathbb{R}^{p})$, we have, by Taylor's expansion
$$f(x_{0}+\delta_{j})=f(x_{0})+\nabla f(x_{0})^{T}\delta_{j}+\frac{1}{2}\delta_{j}^{T}\nabla^{2}f(x_{0})\delta_{j}+\frac{1}{6}\sum_{k_1,k_2,k_3}\left(\nabla^{3}f(x_{0}+t_{j}\delta_{j})\right)_{k_1k_2k_3}(\delta_{j})_{k_1}(\delta_{j})_{k_2}(\delta_{j})_{k_3},$$
for any $j=1,\cdots,n$, where $0\leq t_{j}\leq 1$. Thus under the additional $C^3$ assumption (which we will verify if necessary) the bias of the estimator $L^{p}_{n}$ satisfies
\begin{eqnarray}
&&E(L^{p}_{n})_{i}-(\nabla f(x_{0}))_{i}=f(x_{0})\sum_{j=1}^{n}(w_{j})_{i}+\nabla f(x_{0})^{T}\left(\sum_{j=1}^{n}(w_{j})_{i}\delta_{j}-e_{i}\right)+\sum_{j=1}^{n}\frac{1}{2}(w_{j})_{i}\delta_{j}^{T}\nabla^{2}f(x_{0})\delta_{j}\notag\\
&&+\sum_{j=1}^{n}\frac{1}{6}(w_{j})_{i}\sum_{k_1,k_2,k_3}\left(\nabla^{3}f(x_{0}+t_{j}\delta_{j})\right)_{k_1k_2k_3}(\delta_{j})_{k_1}(\delta_{j})_{k_2}(\delta_{j})_{k_3},\label{biaslinearmulti}
\end{eqnarray}
where $e_{i}$ is the $i$th standard basis in $\mathbb{R}^{p}$. On the other hand, the variance of the estimator $L^{p}_{n}$ is
$$E\|L^{p}_{n}-EL^{p}_{n}\|_{2}^{2}=\sum_{i=1}^{p}Var((L^{p}_{n})_{i})=\sum_{i=1}^{p}\left(\sum_{j=1}^{n}(w_{j})_{i}^{2}\sigma^{2}(x_{0}+\delta_{j})\right),$$
regardless of whether $f(\cdot)\in C^3(\mathbb R^p)$. If $\sum_{j=1}^{n}(w_{j})_{i}\neq0$, we consider the particular cases where $f(\cdot)$ is a constant function, i.e., $f(x)=f(x_0)$ for all $x\in\mathbb{R}^p$. Note that constant functions are in $\mathcal{A}_{q}$ and also in $C^3(\mathbb{R}^p)$. Also note that $(\nabla f(x))_{k_1}=0, (\nabla^{2}f(x))_{k_1k_2}=0, (\nabla^{3} f(x))_{k_1k_2k_3}=0$ for all $x$ and $k_1,k_2,k_3$, therefore from \eqref{biaslinearmulti} we conclude that
$$\sup_{f(\cdot)\in \mathcal{A}_{q}}\left(E(L^{p}_{n})_{i}-(\nabla f(x_{0}))_{i}\right)^{2}\geq\sup_{f(x_0)\in\mathbb{R}}f(x_0)^2\left(\sum_{j=1}^n(w_j)_i\right)^2=\infty.$$
Thus, like in the proof for Theorem \ref{st}, for the purpose of deriving a lower bound for $R_{p}(n,\mathcal{A}_{q},\mathcal{B})$, we can assume without loss of generality that $\sum_{j=1}^{n}w_{j}=0$. Similarly we can assume $\sum_{j=1}^{n}(w_{j})_{i}\delta_{j}-e_{i}=0$. Furthermore, if $\delta_{i}=\delta_{j}$, we assume w.l.o.g. $w_{i}=w_{j}$ since it leads to smaller variance. Now consider $f(\cdot)\in\mathcal{A}_{q}$ such that $f(x_{0}+\delta_{j})=\frac{a}{6}\|\delta_{j}\|_{q}^{3}\cdot sign((w_{j})_{i_{0}})$, and $f(x)=0$ otherwise, where
$$i_{0}=\arg\max_{1\leq i\leq p}\sum_{j=1}^{n}|(w_{j})_{i}|\|\delta_{j}\|^{3}_{q}.$$
In such a case the MSE is bounded from below by
$$\left(\sum_{j=1}^{n}\frac{a}{6}|(w_{j})_{i_{0}}|\|\delta_{j}\|_{q}^{3}\right)^{2}+\sum_{i=1}^{p}\left(\sum_{j=1}^{n}(w_{j})_{i}^{2}\sigma^{2}(x_{0}+\delta_{j})\right).$$
Further considering the case $\sigma^{2}(x_{0}+\delta_{j})=b$, we get
\begin{equation}\label{rep}
\sup_{\substack{f(\cdot)\in\mathcal{A}_q\\\sigma^{2}(\cdot)\in\mathcal{B}}}E\|L^{p}_{n}-\nabla f(x_{0})\|_{2}^{2}\geq\frac{a^{2}}{36}\left(\sum_{j=1}^{n}|(w_{j})_{i_{0}}|\|\delta_{j}\|_{q}^{3}\right)^{2}+b\sum_{j=1}^{n}\|w_{j}\|_{2}^{2}.
\end{equation}
Now since $\sum_{j=1}^{n}(w_{j})_{i}(\delta_{j})_{i}=1$, by H\"older's inequality,
$$p=\sum_{j=1}^{n}w_{j}^{T}\delta_{j}\leq\sum_{j=1}^{n}\|w_{j}\|_{r}^{1/3}\|\delta_{j}\|_{q}\|w_{j}\|_{r}^{2/3}\leq\left(\sum_{j=1}^{n}\|w_{j}\|_{r}\|\delta_{j}\|_{q}^{3}\right)^{1/3}\left(\sum_{j=1}^{n}\|w_{j}\|_{r}\right)^{2/3},$$
where $\frac{1}{q}+\frac{1}{r}=1,q\geq1,r\geq1$.
Thus
\begin{equation}\label{multiineqtotal}
\frac{p^{6}}{\left(\sum_{j=1}^{n}\|w_{j}\|_{r}\right)^{4}}\leq\left(\sum_{j=1}^{n}\|w_{j}\|_{r}\|\delta_{j}\|_{q}^{3}\right)^{2}\leq\left(\sum_{j=1}^{n}\|w_{j}\|_{1}\|\delta_{j}\|_{q}^{3}\right)^{2}\leq p^{2}\left(\sum_{j=1}^{n}|(w_{j})_{i_{0}}|\|\delta_{j}\|_{q}^{3}\right)^{2},
\end{equation}
where the second inequality is due to $r\geq 1$ and the last inequality comes from the definition of $i_{0}$.
Moreover, by H\"older's inequality we have
\begin{equation}\label{multiineqrgeq2}
\left(\sum_{j=1}^{n}\|w_{j}\|_{r}\right)^{2}\leq\left(\sum_{j=1}^{n}\|w_{j}\|_{2}\right)^{2}\leq n\left(\sum_{j=1}^{n}\|w_{j}\|_{2}^{2}\right),\ r\geq2,
\end{equation}
\begin{equation}\label{multiineqr1}
\left(\sum_{j=1}^{n}\|w_{j}\|_{r}\right)^{2}\leq p\left(\sum_{j=1}^{n}\|w_{j}\|_{2}\right)^{2}\leq np\left(\sum_{j=1}^{n}\|w_{j}\|_{2}^{2}\right),\ r=1.
\end{equation}
Therefore, from \eqref{rep}, \eqref{multiineqtotal} and \eqref{multiineqrgeq2} we have
\begin{equation}\label{change}
\sup_{\substack{f(\cdot)\in\mathcal{A}_q\\\sigma^{2}(\cdot)\in\mathcal{B}}}E\|L^{p}_{n}-\nabla f(x_{0})\|_{2}^{2}\geq\frac{a^{2}}{36}\frac{p^{4}}{n^{2}\left(\sum_{j=1}^{n}\|w_{j}\|_{2}^{2}\right)^{2}}+b\sum_{j=1}^{n}\|w_{j}\|_{2}^{2}\geq p^{4/3}\left(\frac{3a^{2}b^{2}}{16}\right)^{1/3}n^{-2/3},\ q=1,2
\end{equation}
where the second inequality is obtained by optimizing over $\sum_{j=1}^n\|w_j\|_2^2$, and its corresponding equality is achieved at
$$\sum_{j=1}^{n}\|w_{j}\|_{2}^{2}=\left(\frac{a^{2}}{b}\right)^{1/3}\frac{p^{4/3}}{18^{1/3}n^{2/3}}.$$
Similarly, from \eqref{rep}, \eqref{multiineqtotal} and \eqref{multiineqr1} we have
\begin{equation}\label{change2}
\sup_{\substack{f(\cdot)\in\mathcal{A}_q\\\sigma^{2}(\cdot)\in\mathcal{B}}}E\|L^{p}_{n}-\nabla f(x_{0})\|_{2}^{2}\geq\frac{a^{2}}{36}\frac{p^{2}}{n^{2}\left(\sum_{j=1}^{n}\|w_{j}\|_{2}^{2}\right)^{2}}+b\sum_{j=1}^{n}\|w_{j}\|_{2}^{2}\geq p^{2/3}\left(\frac{3a^{2}b^{2}}{16}\right)^{1/3}n^{-2/3},\ q=\infty
\end{equation}
where the second inequality is obtained by optimizing over $\sum_{j=1}^n\|w_j\|_2^2$, and its corresponding equality is achieved at
$$\sum_{j=1}^{n}\|w_{j}\|_{2}^{2}=\left(\frac{a^{2}}{b}\right)^{1/3}\frac{p^{2/3}}{18^{1/3}n^{2/3}}.$$
Since $\delta_{j},w_{j}$ are arbitrary, we conclude from \eqref{change} and \eqref{change2} that
$$\inf_{\substack{\delta_{j},j=1,\cdots,n\\w_{j},j=1,\cdots,n}}\sup_{\substack{f(\cdot)\in\mathcal{A}_q\\\sigma^{2}(\cdot)\in\mathcal{B}}}E\|L^{p}_{n}-\nabla f(x_{0})\|_{2}^{2}\geq p^{4/3}\left(\frac{3a^{2}b^{2}}{16}\right)^{1/3}n^{-2/3},\ q=1,2,$$
$$\inf_{\substack{\delta_{j},j=1,\cdots,n\\w_{j},j=1,\cdots,n}}\sup_{\substack{f(\cdot)\in\mathcal{A}_q\\\sigma^{2}(\cdot)\in\mathcal{B}}}E\|L^{p}_{n}-\nabla f(x_{0})\|_{2}^{2}\geq p^{2/3}\left(\frac{3a^{2}b^{2}}{16}\right)^{1/3}n^{-2/3},\ q=\infty.$$
This completes the first assertion of the theorem.
\\

\noindent\emph{Proof of Assertion 2.} Suppose the budget $n$ is a multiple of $2p$. We allocate $n/p$ budget to each dimension. Then for any $f(\cdot)\in\mathcal{A}_{q}$ and $\sigma^{2}(\cdot)\in\mathcal{B}$, from \eqref{ub} in the proof of Theorem~\ref{st} we get
$$E\left((\bar{L}^{p}_{n})_{i}-(\nabla f(x_{0}))_{i}\right)^{2}\leq\left(\frac{3a^{2}b^{2}}{16}\right)^{1/3}\left(\frac{n}{p}\right)^{-2/3}=p^{2/3}\left(\frac{3a^{2}b^{2}}{16}\right)^{1/3}n^{-2/3}.$$
Thus
$$E\|\bar{L}_{n}^{p}-\nabla f(x_{0})\|^{2}_{2}\leq p^{5/3}\left(\frac{3a^{2}b^{2}}{16}\right)^{1/3}n^{-2/3}.$$

\noindent\emph{Proof of Assertion 3.} If each $w_{j}$ has the same sign across components, instead of~(\ref{rep}), we have the sharper lower bound
$$\frac{a^{2}}{36}\sum_{i=1}^{p}\left(\sum_{j=1}^{n}|(w_{j})_{i}|\|\delta_{j}\|_{q}^{3}\right)^{2}+b\sum_{j=1}^{n}\|w_{j}\|_{2}^{2},$$
and note that we also have H\"older's inequality
$$\left(\sum_{j=1}^{n}\|w_{j}\|_{1}\|\delta_{j}\|_{q}^{3}\right)^{2}\leq p\sum_{i=1}^{p}\left(\sum_{j=1}^{n}|(w_{j})_{i}|\|\delta_{j}\|_{q}^{3}\right)^{2}.$$ Thus the $p^{4/3}$ factor in~(\ref{change}) can be improved to $p^{5/3}$, and the $p^{2/3}$ factor in~(\ref{change2}) can be improved to $p$.
Since each linear coefficient $w_j$ in the estimator $\bar{L}_n^p$ has the same sign across components, combining with \eqref{me} we conclude that $\bar L_n^p$ is exactly optimal for the restricted class of linear estimators, when considering function class $\mathcal{A}_{q}$ with $q=1,2$.
\end{proof}

\begin{proof}[Proof of Theorem \ref{general multi thm}]
Like in the proof of Theorem \ref{general thm}, for convenience, we set $x_0=0$ w.l.o.g. in this proof, so that
\begin{equation}\label{classgeneralmulti}
\mathcal{A}_{q}=\left\{f(\cdot):\nabla^{2} f(0)\textrm{ exists and }\left|f(x)-f(0)-\nabla f(0)^{T}x-\frac{1}{2}x^{T}\nabla^{2}f(0)x\right|\leq \frac{a}{6}\|x\|^{3}_{q}\right\}
\end{equation}
Following Le Cam's method in the proof of Theorem~\ref{general thm} and replacing $|f_{1}^{'}(0)-f_{2}^{'}(0))|=\epsilon$ by $\left\|\nabla f_{1}(0)-\nabla f_{2}(0)\right\|_{2}=\epsilon$ for the multi-dimensional setting, we have (see \eqref{stdrelation} and \eqref{klexpr} in the proof of Theorem 4), for any $f_{1},f_{2}\in\mathcal{A}_{q}$,
\begin{equation}\label{stdrelationmulti}
R_{p}(n,\mathcal{A}_{q},\mathcal{B})\geq\frac{\left\|\nabla f_{1}(0)-\nabla f_{2}(0)\right\|_{2}^{2}}{16}e^{-\dfrac{\left\|\mu_{2}-\mu_{1}\right\|_{2}^{2}}{2b}},
\end{equation}
where $\mu_{k}=(f_{k}(x_{1}),\ldots,f_{k}(x_{n})),k=1,2$.

In order to maximize the lower bound, let us define the modulus function~\cite{donoho1994statistical}
for function class $\mathcal{A}_{q}$:
\begin{gather*}
\epsilon_{\mathcal{A}_{q}}(\omega)=\sup\left\{\|\nabla f_{1}(0)-\nabla f_{2}(0)\|_{2}: \underset{x}{\sup}\left|f_{1}(x)-f_{2}(x)\right|\leq\omega,\ f_{1},f_{2}\in\mathcal{A}_{q}\right\},
\end{gather*}
which is not only a function of $\omega$ (like that in Theorem~\ref{general thm}) but also affected by the choice of $q$. Here the extremal pair ($f_{1},f_{2}$) attaining the modulus function will be different for different $q$. First, by Lemma 7 of~\cite{donoho1991geometrizing} (which we include in Appendix~\ref{appendlemma} for self-contained purpose), the extremal pair can be chosen of the form: $f_{1}=f$ and $f_{2}=-f$. Thus
\begin{gather*}
\epsilon_{\mathcal{A}_{q}}(\omega)=2\sup\left\{\|\nabla f(0)\|_{2}: \underset{x}{\sup}\left|f(x)\right|\leq\omega/2,\ f\in\mathcal{A}_{q}\right\}.
\end{gather*}
It follows that $\epsilon_{\mathcal{A}_{q}}$ is the inverse function of
\begin{gather}
\omega_{\mathcal{A}_{q}}(\epsilon)=2\inf\left\{\underset{x}{\sup}\left|f\right|:\|\nabla f(0)\|_{2}=\epsilon/2,\ f\in\mathcal{A}_{q}\right\}.\label{problem}
\end{gather}
If $f(x)$ solves problem~(\ref{problem}), so does $-f(-x)$. As the norm is a convex function, $(f(x)-f(-x))/2$ is then also a solution. Therefore, we can restrict attention to odd functions in our search for a solution to~(\ref{problem}).

For each odd function $f\in\mathcal{A}_{q}$,
$$\left|f(x)-\nabla f(0)^{T}x\right|=\left|f(x)-f(0)-\nabla f(0)^{T}x-\frac{1}{2}x^{T}\nabla^{2}f(0)x\right|\leq\frac{a}{6}\|x\|^{3}_{q}.$$
It follows that
$$\left|f(x)\right|\geq \left|\nabla f(0)^{T}x\right|-\left|f(x)-\nabla f(0)^{T}x\right|\geq\left|\nabla f(0)^{T}x\right|-\frac{a}{6}\|x\|^{3}_{q}.$$
Therefore
$$\sup_x |f(x)|\geq \sup_x \left[|\nabla f(0)^T x|-\frac{a}{6}\|x\|_q^3\right]_+.$$
Denote $X_+ =\{ x:\nabla f(0)^Tx\geq0\}$ and $X_- = \{x:\nabla f(0)^Tx<0\}$. We note that the odd function
$$g(x)  = \left[ \nabla f(0)^T x-\frac{a}{6}\|x\|_q^3\right]_+\cdot1_{x\in X_+}  + \left[ \nabla f(0)^T x+\frac{a}{6}\|x\|_q^3\right]_+\cdot1_{x\in X_-} $$
belongs to $\mathcal{A}_q$, $\sup_x|f(x)|\geq\sup_x|g(x)|$, and also that $\nabla g(0)=\nabla  f(0)$. Therefore, we consider functions of the form
$$f(x)  = \left[ \xi^T x-\frac{a}{6}\|x\|_q^3\right]_+\cdot1_{\xi^T x\geq0}  + \left[ \xi^T x+\frac{a}{6}\|x\|_q^3\right]_+\cdot1_{\xi^T x<0} $$
with $\|\xi\|_2 = \|\nabla f(0)\|_2=\epsilon/2$ in searching for a solution to~(\ref{problem}). Moreover, if $\xi^{*}$ is a solution to the following problem
\begin{equation}\label{eq:optopt}
\min_{\|\xi\|_2=\epsilon/2}\max_{\|x\|_q=1,\xi^Tx\geq0} \xi^Tx.
\end{equation}
then we have that
$$f^{*}(x) = \left[ (\xi^{*})^T x-\frac{a}{6}\|x\|_q^3\right]_+\cdot1_{(\xi^{*})^T x\geq0}  + \left[ (\xi^{*})^T x+\frac{a}{6}\|x\|_q^3\right]_+\cdot1_{(\xi^{*})^T x<0} $$
is a solution to~(\ref{problem}).

For $q=1$ and $2$, it is easy to verify that
$$\xi^{*} = \left(\frac{\epsilon}{2\sqrt{p}},\ldots,\frac{\epsilon}{2\sqrt{p}}\right)$$
is a solution to~(\ref{eq:optopt}). Therefore,
\begin{gather*}
f^{*}(x)=sign\left(\sum_{i=1}^{p}(x)_{i}\right)\left[\frac{\epsilon}{2\sqrt{p}}\left|\sum_{i=1}^{p}(x)_{i}\right|-\frac{a}{6}\|x\|_{q}^{3}\right]_{+}.
\end{gather*}
Moreover, we have
\begin{gather*}
\sup_x|f^{*}(x)| = \sup_{t\geq0}\left[\frac{\epsilon}{2}\sqrt{p}t-\frac{a}{6}p^{3}t^3\right]_+ =\frac{\epsilon}{3}\sqrt{\frac{\epsilon}{a}}p^{-3/4}, \text{\ for\ } q=1,\\
\sup_x|f^{*}(x)| = \sup_{t\geq0}\left[\frac{\epsilon}{2}\sqrt{p}t-\frac{a}{6}p^{3/2}t^3\right]_+ =\frac{\epsilon}{3}\sqrt{\frac{\epsilon}{a}}, \text{\ for\ } q=2.
\end{gather*}


For $q=\infty$, it is easy to verify that
$$\xi^{*}=\left(\frac{\epsilon}{2},0,\ldots,0\right)$$ is a solution to~(\ref{eq:optopt}). Therefore,
\begin{gather*}
f^{*}(x)=sign\left((x)_1\right)\left[\frac{\epsilon}{2}\left|(x)_{1}\right|-\frac{a}{6}\|x\|_{q}^{3}\right]_{+}.
\end{gather*}
Moreover, we have
$$\sup_x|f^{*}(x)| = \sup_{t\geq0}\left[\frac{\epsilon}{2}t-\frac{a}{6}t^3\right]_+ =\frac{\epsilon}{3}\sqrt{\frac{\epsilon}{a}}.$$
Thus
\begin{gather*}
\omega_{\mathcal{A}_{q}}(\epsilon)=2\sup_x|f^{*}(x)|=
\frac{2\epsilon}{3}\sqrt{\frac{\epsilon}{a}}p^{-3/4}, \text{\ for\ } q=1,\\
\omega_{\mathcal{A}_{q}}(\epsilon)=2\sup_x|f^{*}(x)|=\frac{2\epsilon}{3}\sqrt{\frac{\epsilon}{a}}, \text{\ for\ } q=2,\infty.
\end{gather*}
Therefore, from \eqref{stdrelationmulti} we have
\begin{eqnarray*}
R_{p}(n,\mathcal{A}_{q},\mathcal{B})\geq\frac{\epsilon^{2}}{16}e^{-\dfrac{n\omega_{\mathcal{A}_{q}}^{2}(\epsilon)}{2b}}=\frac{\epsilon^{2}}{16}e^{-\dfrac{2n\epsilon^{3}}{9abp^{3/2}}}, \text{\ for\ } q=1,\\
R_{p}(n,\mathcal{A}_{q},\mathcal{B})\geq\frac{\epsilon^{2}}{16}e^{-\dfrac{n\omega_{\mathcal{A}_{q}}^{2}(\epsilon)}{2b}}=\frac{\epsilon^{2}}{16}e^{-\dfrac{2n\epsilon^{3}}{9ab}}, \text{\ for\ } q=2,\infty.
\end{eqnarray*}
Further by choosing
$$\epsilon=\left(\dfrac{3abp^{3/2}}{n}\right)^{1/3}, \text{\ for\ } q=1,$$
$$\epsilon=\left(\dfrac{3ab}{n}\right)^{1/3}, \text{\ for\ } q=2,\infty,$$
at which $\frac{\epsilon^{2}}{16}e^{-\frac{2n\epsilon^{3}}{9abp^{3/2}}}$ and $\frac{\epsilon^{2}}{16}e^{-\frac{2n\epsilon^{3}}{9ab}}$ achieve their minima respectively,
we get
$$R_{p}(n,\mathcal{A}_{q},\mathcal{B})\geq\frac{1}{16}e^{-2/3}\left(\dfrac{3abp^{3/2}}{n}\right)^{2/3}\approx 0.0667p\left(\frac{ab}{n}\right)^{2/3}, \text{\ for\ } q=1,$$
$$R_{p}(n,\mathcal{A}_{q},\mathcal{B})\geq\frac{1}{16}e^{-2/3}\left(\dfrac{3ab}{n}\right)^{2/3}\approx 0.0667\left(\frac{ab}{n}\right)^{2/3}, \text{\ for\ } q=2,\infty.$$
\end{proof}

\end{document}